\documentclass[oneside,english]{amsart}
\usepackage[T1]{fontenc}
\usepackage[latin9]{inputenc}
\usepackage{amstext}
\usepackage{amsthm}
\usepackage{amssymb}
\usepackage{graphicx}

\makeatletter
\numberwithin{equation}{section}
\numberwithin{figure}{section}
\theoremstyle{plain}
\newtheorem{thm}{\protect\theoremname}
\theoremstyle{definition}
\newtheorem{defn}[thm]{\protect\definitionname}
\theoremstyle{plain}
\newtheorem{cor}[thm]{\protect\corollaryname}
\theoremstyle{plain}
\newtheorem{lem}[thm]{\protect\lemmaname}

\usepackage{cite}

\usepackage{babel}
\providecommand{\corollaryname}{Corollary}
  \providecommand{\definitionname}{Definition}
  \providecommand{\lemmaname}{Lemma}
  
\providecommand{\theoremname}{Theorem}

\makeatother

\usepackage{babel}
\providecommand{\corollaryname}{Corollary}
\providecommand{\definitionname}{Definition}
\providecommand{\lemmaname}{Lemma}
\providecommand{\theoremname}{Theorem}

\begin{document}
\title{Absolute stability and absolute hyperbolicity in systems with discrete
time-delays}
\author{Serhiy Yanchuk, Matthias Wolfrum, Tiago Pereira, Dmitry Turaev}
\begin{abstract}
An equilibrium of a delay differential equation (DDE) is absolutely
stable, if it is locally asymptotically stable for all delays. We
present criteria for absolute stability of DDEs with discrete time-delays.
In the case of a single delay, the absolute stability is shown to
be equivalent to asymptotic stability for sufficiently large delays.
Similarly, for multiple delays, the absolute stability is equivalent
to asymptotic stability for hierarchically large delays. Additionally,
we give necessary and sufficient conditions for a linear DDE to be
hyperbolic for all delays. The latter conditions are crucial for determining
whether a system can have stabilizing or destabilizing bifurcations
by varying time delays. 
\end{abstract}

\maketitle
\tableofcontents{}

\section{Introduction}

Delay differential equations (DDE) play an important role in modeling
various processes in nature and technology. Examples are optoelectronic
systems \cite{Erneux2009,Weicker2012,Soriano2013,Glitzky2014,Yanchuk2014,Erneux2017,Yanchuk2017},
population and infections disease modeling \cite{Kuang1993,Smith1993,Hartung2006a,Diekmann2007,Erneux2009,Smith2010,Muller2015,Agaba2017,Young2019,Ruschel2019},
neuroscience \cite{Wu2001,Izhikevich2006,Ko2007,Popovych2011,Yanchuk2019},
machine learning \cite{Appeltant2011,Larger2012,Larger2017,Stelzer2019,Goldmann2020},
mechanics \cite{Tsao1993,Hartung2006a,Insperger2006,InspergerStepan2011,Otto2013,Otto2015},
and other fields. Driven by industrial developments and automatic
control devices, DDE theory was rapidly developing since the middle
of the 20th century \cite{Pontryagin1942,Myshkis1949,Bellman1954}.
Several monographs have been published, see, for example, \cite{Bellman1963a,Stepan1989,Hale1993,Diekmann1995a,Erneux2009,Atay2010a,GuoWu2013}.

It is a basic fact that the equilibria of a DDE do not change under
variations of the delay time. In general, their stability properties
may change under such variations. Indeed, in many cases increasing
delay is known to induce additional instabilities. However, there
is also the case, called \emph{absolute stability}, where the stability
of an equilibrium remains unchanged for all possible non-negative
delay times. Considering linear DDEs with discrete delays 
\begin{equation}
\frac{dx}{dt}(t)=A_{0}x(t)+\sum_{k=1}^{m}A_{k}x(t-\tau_{k}),\label{eq:lin-1}
\end{equation}
with $x\in\mathbb{R}^{n}$, $\tau_{k}\ge0$, $A_{0},A_{k}\in\mathbb{C}^{n\times n}$,
$k=1,\dots,m$. System (\ref{eq:lin-1}) is the linearization at an
equilibrium of autonomous DDEs. The stability of DDE (\ref{eq:lin-1})
is described by the roots of the characteristic quasipolynomial 
\begin{equation}
Q(\lambda)=P(\lambda,e^{-\lambda\tau_{1}},\dots,e^{-\lambda\tau_{k}})=\det\left[\lambda\cdot\mathrm{I}-A_{0}-\sum_{k=1}^{m}A_{k}e^{-\lambda\tau_{k}}\right]=0,\label{eq:cheq-1}
\end{equation}
where $I$ is the identity matrix.

We present a new criterion for the absolute stability of Eq.~(\ref{eq:lin-1}),
i.e., a necessary and sufficient condition on the matrices $A_{k}$
such that all roots $\lambda$ of the quasipolynomial (\ref{eq:cheq-1})
have negative real parts for arbitrary non-negative delays $\tau_{k}$.
Our Theorems \ref{thm:AS-general} and \ref{thm:AS-general-1} generalize
known results \cite{Pontryagin1942,Yuang-Shun,Bellman1963a,Elsgolz1971a,Cooke1986,Brauer1987,ShiguiRuan2003,Hayes1950,Noonburg1969,Boese1995,Baptistini1997,Boese1998,Kolmanovskii1999,Wang1999,Wu2001a,Smith2010,Li2016,An2019}
and have three main advantages: 
\begin{itemize}
\item simple to check (conditions on compact sets);
\item they give necessary and sufficient conditions; 
\item geometric interpretation using certain limiting spectral sets. 
\end{itemize}
Moreover, the absolute stability appears to be equivalent to the asymptotic
stability for hierarchically large delays $1\ll\tau_{1}\ll\cdots\ll\tau_{m}$,
which, for the case $m=1$, is the asymptotic stability for a single
large delay.

Additionally, we provide a criterion for system (\ref{eq:lin-1})
to be hyperbolic for all time delays, i.e., the condition for the
absence of the roots of the characteristic polynomial $\lambda$ with
zero real parts. In particular, this means that under the obtained
conditions one cannot change the stability of the equilibrium in (\ref{eq:lin-1});
it remains either asymptotically stable or unstable for all delays.

Let us first give a brief overview of the known results on the absolute
stability. One of the first conditions is due to Pontryagin \cite{Pontryagin1942}.
This criterium involves the verification of certain properties of
the characteristic equation evaluated along the whole imaginary axis
$\Delta(iy)$ as well as some additional implicit conditions. The
Potryagin conditions have been used in many applications \cite{Bellman1963a,Hale1977}.

In \cite{Brauer1987}, Brauer gave sufficient conditions for the absolute
stability of the characteristic equation 
\begin{equation}
F(\lambda)+G(\lambda)e^{-\lambda\tau}=0,\label{eq:brauercheq}
\end{equation}
which is a polynomial of the first order in $e^{-\lambda\tau}$. Comparing
it with (\ref{eq:cheq-1}), this corresponds to a single delay and
a rank one matrix $A_{1}$. On the other hand, equation (\ref{eq:brauercheq})
can appear in some cases with distributed delays, which we do not
consider here. The Brauer's conditions have been applied in, e.g.
\cite{Smith2010,Cooke1986}.

Cooke and van der Driessche also considered Eq.~(\ref{eq:brauercheq})
as well as a generalization to multiple delays in \cite{Cooke1986};
they provided sufficient conditions for the absolute stability. Chin
Yuang-Shun \cite{Yuang-Shun} gave criterion for the case of one delay.
This criterion requires $Q(iy)\ne0$ for all $y\in\mathbb{R}$ and
all $\tau_{1}\ge0$, which includes the time-delay as a parameter.
Instead, a practically employable criterion for absolute stability
in the case of a single delay should be delay-independent and given
by an at most one-parameter condition. In section~\ref{sec:one-delay},
we provide such a criterion and explain its geometrical meaning. The
Pontryagin type conditions, in contrast, are hard to check, and in
the case of multiple delays they are very laborious. 

Several other sufficient conditions are given in \cite{Elsgolz1971a,Kolmanovskii1999},
for the case of two delays in \cite{ShiguiRuan2003}, neutral equations
in \cite{Boese1991}, and some special types of equations in \cite{Hayes1950,Noonburg1969,Boese1995,Baptistini1997,Boese1998,An2019}.
In \cite{Chen1995a,Li2016}, a strong delay-independent stability
is used to give sufficient conditions for the absolute stability,
which is called there weak delay-independent stability. Applications
to control problems are considered in \cite{Wang1999,Wu2001a}.

\section{General criterion for absolute stability}

First, we introduce some notation and definitions. Our notation is
that of Ref. \cite{Hale1993}. Given a bounded linear operator $A$,
its \emph{spectrum} is denoted by $\sigma(A)$ and its \emph{spectral
radius} is denoted by $\rho(A)$. An $n\times n$ matrix $A$ is \emph{Hurwitz}
if $\Re\sigma(A)<0$.

Given a finite family of operators $A_{k}:\mathbb{C}^{n}\rightarrow\mathbb{C}^{n}$
for $k=\{0,1,\dots,m\}$ of Eq. (\ref{eq:lin-1}), we consider feedback
phases $\Phi=(\varphi_{1},\dots,\varphi_{m})\in\mathbb{T}^{m}$ and
\[
S(\Phi)=A_{0}+\sum_{k=1}^{m}A_{k}e^{i\varphi_{k}}.
\]

Our key object is the phase dependent spectrum $\sigma(S(\Phi))\subset\mathbb{C}$,
which will contain key information about the stability of the system.
\begin{defn}
System (\ref{eq:lin-1}) is \emph{absolutely stable} if all roots
$\lambda$ of the characteristic equation (\ref{eq:cheq-1}) possess
negative real parts $\Re\left(\lambda\right)<0$ for all $\tau_{k}\ge0$,
$k=1,\dots,m$. Similarly, we call (\ref{eq:lin-1}) \emph{absolutely
hyperbolic} if all roots have nonzero real parts for all delays. 
\end{defn}

As follows from the general DDE theory \cite{Hale1993}, in case of
absolute stability, all solutions of the initial value problem for
DDE (\ref{eq:lin-1}) are exponentially asymptotically stable, i.e.
$x(t;\varphi)\to0$ exponentially fast with $t\to\infty$ for any
initial function $\varphi(\theta)=x(\theta;\varphi)$, $\theta\in[-\max_{k}\tau_{k},0]$.

The following theorem provides a general criterion for the absolute
stability in the case of multiple discrete delays. 
\begin{thm}
\label{thm:AS-general} System (\ref{eq:lin-1}) is absolutely stable
if and only if the following conditions are satisfied:\\
\textup{(A1.1) {[}}\emph{instantaneous stability}\textup{{]}:} $A_{0}$
is Hurwitz.\\
\textup{(A1.2) {[}nonsingular $S(0)${]}:} $S(0)$ is nonsingular.
\\
\textup{(A1.3) {[}no resonance{]}:} $i\omega\notin\sigma(S(\Phi))$
for all $\Phi\in\mathbb{T}^{m}$ and $\omega\ne0$.

Moreover, the conditions \textup{(A1.2) }and \textup{(A1.3) }are necessary
and sufficient for system (\ref{eq:lin-1}) to be absolutely hyperbolic.
\end{thm}

Let us discuss the meaning of the above conditions. Condition (A1.1)
{[}instantaneous stability{]} means that the corresponding instantaneous
ODE system $\dot{x}=A_{0}x$ must be exponentially stable. Condition
(A1.2) {[}nonsingular $S(0)${]} is equivalent to the requirement
that the characteristic quasipolynomial \eqref{eq:cheq-1} does not
possess a zero root. We will later show that, taking into account
(A1.1) {[}instantaneous stability{]} and (A1.3) {[}no resonance{]},
the condition (A1.2) can be replaced by the requirement that $S(0)$
is Hurwitz. Hence, (A1.2) {[}nonsingular $S(0)${]} contributes to
the exponential stability of the ODE system $\dot{x}=S(0)x$ obtained
from \eqref{eq:lin-1} for zero delays.

Condition (A1.3) {[}no resonance{]} means that the spectrum of the
$m$-parametric set of matrices $S(\Phi)$ cannot cross the imaginary
axis apart from the origin. We will show later that, taking into account
(A1.1) {[}instantaneous stability{]}, the condition (A1.3) is equivalent
of having $S(\Phi)$ ``almost Hurwitz'', i.e., $\Re\sigma\left(S(\Phi)\right)<0$
except that the possible zero eigenvalue. We will also show that $\sigma\left(S(\Phi)\right)$
can be in a certain sense related to the asymptotic spectrum in delay
systems with hierarchically long delays. Moreover, purely imaginary
eigenvalues $i\omega$ of $\sigma(S(\Phi))$, which we call \emph{resonances},
appear as characteristic roots of \eqref{eq:cheq-1} at an infinite
sequence of resonant delay times. 

Moreover, purely imaginary values $i\omega\in\sigma(S(\Phi))$ correspond
to certain ''resonances'' and the appearance of critical characteristic
roots for countable number of delays.

The three conditions (A1.1), (A1.2), and (A1.3) are finite-dimensional
problems involving the calculation of the spectrum of some $n\times n$
matrices. The condition (A1.3) {[}no resonance{]} contains a compact
$m$-parameter family of matrices. 

The conditions for absolute stability can be equivalently formulated
as follows.
\begin{thm}
\label{thm:AS-general-1} System (\ref{eq:lin-1}) is absolutely stable
if and only if the following conditions are satisfied:\\
\textup{(A1.2) }\emph{{[}nonsingular $S(0)${]}}\textup{:} $S(0)$
is nonsingular. \\
\textup{(A2.2) {[}almost Hurwitz S($\Phi$){]}: }$S(\Phi)$ is Hurwitz,
except for a possible zero eigenvalue.
\end{thm}

The proof will be given in Sec.~\ref{sec:Proofs}.

Combining the asymptotic spectral theory from \cite{Lichtner2011,Sieber2013}
for the case of one delay with Theorem \ref{thm:AS-general}, we can
show that the absolute stability is determined by the stability at
large delays. In particular, we obtain the following 
\begin{cor}
\label{thm:absolut-vs-large} System (\ref{eq:lin-1}) with one delay
is absolutely stable if and only if it is asymptotically exponentially
stable for all sufficiently large delays, i.e. there exists $\tau_{L}$
such that $\Re(\lambda)<0$ for all characteristic roots and all $\tau>\tau_{L}$. 
\end{cor}

In fact, Corollary \ref{thm:absolut-vs-large} is a consequence of
the following more general statement for the case of multiple delays. 
\begin{thm}
\label{thm:absolute-vs-hierar} System (\ref{eq:lin-1}) is absolutely
stable if and only if the system with hierarchical time delays 
\begin{equation}
\tau_{1}=\varepsilon^{-1},\quad\tau_{k}=\nu_{k}\varepsilon^{-k},\quad k=2,\dots,m,\label{eq:delays-hier}
\end{equation}
is asymptotically exponentially stable for all sufficiently small
$\varepsilon\ll1$ and all $\nu_{k}\in[1,1+\varepsilon^{k-1})$. 
\end{thm}

The stability for one large delay has a useful interpretation from
the point of view of a singular map. By rescaling the time $t=T/\varepsilon$
with $\varepsilon=1/\tau$, we obtain 
\begin{equation}
\varepsilon\dot{x}(T)=A_{0}x(T)+A_{1}x(T-1).\label{eq:DDE-1delay-eps}
\end{equation}
By neglecting formally the left-hand side, we obtain the singular
map 
\begin{equation}
x(T)=-A_{0}^{-1}A_{1}x(T-1),\label{eq:map}
\end{equation}

This hints that the stability of the system can be obtained at a formal
level by a discrete dynamical system. There are many publications
devoted to relations between the DDE (\ref{eq:DDE-1delay-eps}) and
the singular map (\ref{eq:map}), see \cite{Mallet-Paret1986,Chow1989,Ivanov1989,Mallet-Paret1989,Hale1993,Hale1996,Huang2000,Yanchuk2005b,Pellegrin2014a,Yanchuk2015b,Ruschel2017}.
In fact, in order to obtain equivalent stability conditions, one should
consider an extended singular map 
\begin{equation}
x(T)=\left(i\omega\text{I}-A_{0}\right)^{-1}A_{1}e^{i\varphi}x(T-1).\label{eq:map-omega}
\end{equation}
We will provide a discussion about this form in Sec. \ref{sec:singular}.
Using this dynamical system we can conclude absolute stability as
shown in the following 
\begin{cor}
\label{cor:singular} System (\ref{eq:lin-1}) for one delay is absolutely
stable if and only if 
\begin{itemize}
\item $A_{0}$ is Hurwitz; 
\item the discrete dynamical system (\ref{eq:map-omega}) is asymptotically
exponentially stable for $\omega\ne0$; 
\item for $\omega=0$, the discrete dynamical system (\ref{eq:map-omega})
possesses multipliers $\mu$ with $|\mu|\le1$ and $\mu\ne1$, i.e.,
it is either asymptotically exponentially stable or neutral with $\mu=e^{i\varphi}$,
$\varphi\ne2\pi k$. 
\end{itemize}
\end{cor}

\noindent \textbf{Organisation of the manuscript.} We provide examples
of the application of Theorem~\ref{thm:AS-general} to scalar DDE
with multiple delays in Sec. \ref{scalarDDE} and give a geometric
interpretation of the obtained criterion for one delay in a system
of DDE's in Sec. \ref{sec:one-delay} emphasising the role of asymptotic
spectrum for large delays. We consider the case of multiple hierarchical
delays in Sec. \ref{multipleDelays}. We offer proofs of Theorems~\ref{thm:AS-general}
and \ref{thm:AS-general-1} in Sec \ref{sec:Proofs}. Finally, we
provide conclusions and some open problems in Sec. \ref{con}.

\section{Scalar DDEs}

\label{scalarDDE}

In the case of scalar DDEs
\begin{equation}
\dot{x}(t)=a_{0}x(t)+\sum_{k=1}^{m}a_{k}x(t-\tau_{k}),\quad a_{j}\in\mathbb{C},j=1,\dots,m,\label{eq:scalar-2delays}
\end{equation}
 the absolute stability conditions can be significantly simplified. 
\begin{cor}
\label{thm:Scalar} System (\ref{eq:scalar-2delays}) is absolutely
stable if and only if the following conditions are satisfied 
\begin{equation}
\Re\left(a_{0}\right)+\sum_{k=1}^{m}|a_{k}|<0\,\,\,\,\text{for}\,\,\,\,\Im(a_{0})\ne0,\label{eq:kdelays-con1}
\end{equation}
\begin{equation}
a_{0}+\sum_{k=1}^{m}|a_{k}|\le0\,\,\,\,\text{and}\,\,\,\,\sum_{k=0}^{m}a_{k}\ne0\,\,\,\text{for \,\,\,}\Im(a_{0})=0.\label{eq:k-delays-cond2}
\end{equation}
\end{cor}

\begin{proof}
We verify that the conditions of Theorem \ref{thm:AS-general-1} are
equivalent to (\ref{eq:kdelays-con1})--(\ref{eq:k-delays-cond2}).
In order to simplify the condition (A2.2) {[}almost Hurwitz S($\Phi$){]}
for the scalar case, we observe that the maximum of the real part
of $a_{0}+\sum_{k=1}^{m}a_{k}e^{i\varphi_{k}}$ is achieved at $\varphi_{k}=-\arg a_{k}$,
$k=1,\dots,m$, and it equals 
\begin{equation}
\max_{\varphi_{1,},\dots,\varphi_{m}}\left(\Re\left(a_{0}+\sum_{k=1}^{m}a_{k}e^{i\varphi_{k}}\right)\right)=\Re\left(a_{0}\right)+\sum_{k=1}^{m}|a_{k}|.\label{eq:max}
\end{equation}

For $\Im(a_{0})\ne0$, this isolated maximum has nonzero imaginary
part and must be negative accordingly to (A2.2). Therefore, we obtain
(\ref{eq:kdelays-con1}) with strict inequality as an equivalent to
(A2.2). 

For $\Im(a_{0})=0$, the maximum (\ref{eq:max}) is $a_{0}+\sum_{k=1}^{m}\left|a_{k}\right|$.
As zero is allowed accordingly to the condition (A2.2) {[}almost Hurwitz
S($\Phi$){]}, we obtain non-strict inequality in (\ref{eq:k-delays-cond2}). 

Finally, we observe that $\sum_{k=0}^{m}a_{k}\ne0$ is equivalent
to (A1.2) {[}nonsingular $S(0)${]}. This inequality must be added
in (\ref{eq:k-delays-cond2}) only, since $\sum_{k=0}^{m}a_{k}\ne0$
is satisfied under the condition (\ref{eq:kdelays-con1}). 
\end{proof}
Numerical examples with scalar DDEs will be presented in Secs.~\ref{subsec:Scalar-DDEs-with}
and \ref{subsec:Illustration-in-the}.

\section{The case of one delay, geometric interpretation\label{sec:one-delay}}

Since the case of one discrete delay appears most often in applications,
we discuss it here in more detail. In particular, we give a geometric
interpretation using the asymptotic spectrum for large delay. 

\subsection{Auxiliary results}

The following technical Lemmas will be needed.
\begin{lem}
Let $A,B\in\mathbb{C}^{n\times n}$. If $A+Be^{i\varphi}$ is Hurwitz
for all $\varphi\in\mathbb{T}$, then $A$ is Hurwitz. 
\end{lem}

\begin{proof}
Assume the opposite, that is $\lambda_{0}\in\sigma(A)$ with $\Re\left(\lambda_{0}\right)\ge0$.
Consider the function
\[
P(\lambda,z)=\det\left(-\lambda\text{I}+A+zBe^{i\varphi}\right),
\]
which is a polynomial in $\lambda$. There exists a continuous branch
of complex roots $\lambda(z)$ of this polynomial such that $\lambda(0)=\lambda_{0},$$\Re(\lambda(0))\ge0$
and $\Re(\lambda(1))<0$. Due to continuity, there exists a real number
$\hat{z}\in[0,1)$ such that $\lambda(\hat{z})=i\hat{\omega}$. Hence,
we have $P(i\hat{\omega},\hat{z})=0$. Consider $P(i\omega,z)=0$
as a polynomial in $z$. If this polynomial depends trivially on $z$
at $\omega=\hat{\omega}$, then $P(i\hat{\omega},1)=0$ and we immediately
obtain the contradiction to the Hurwitz property of $A+Be^{i\varphi}$.
If $P(i\omega,z)$ is a nontrivial polynomial in $z$ at $\omega=\hat{\omega}$,
then there exists a continuous branch of complex roots $z(\omega)$
such that $z(\hat{\omega})=\hat{z}$, $\left|z(\hat{\omega})\right|<1$,
and $\left|z(\omega)\right|\to\infty$ as $\omega\to\infty$. Hence,
there exists $\tilde{\omega}>\hat{\omega}$ such that $\left|z(\tilde{\omega})\right|=1$.
This means that $P(i\tilde{\omega},e^{i\arg z(\tilde{\omega})})=0$,
and the matrix $A+Be^{i\left(\varphi+\arg z(\tilde{\omega})\right)}$
is not Hurwitz. The contradiction proves the Lemma. 
\end{proof}
\begin{lem}
\label{lem:new-lemma-1}Let $A\in\mathbb{C}^{n\times n}$ be Hurwitz.
Then, for any $B\in\mathbb{C}^{n\times n}$, one of the following
three mutually exclusive cases occurs: \\
I. $A+Be^{i\varphi}$ is Hurwitz for all $\varphi\in\mathbb{T}$;\\
II. There exist $\tilde{\omega}\ne0$ and $\tilde{\varphi}$ such
that $i\tilde{\omega}\in\sigma\left(A+Be^{i\tilde{\varphi}}\right)$;\\
III. There exist one or several values $\tilde{\varphi}_{1},\dots,\tilde{\varphi}_{l}$
($l\le n$) such that $0\in\sigma\left(A+Be^{i\tilde{\varphi}_{j}}\right)$,
$j=1,\dots,l$, and $A+Be^{i\varphi}$ is Hurwitz for all $\varphi\ne\tilde{\varphi}_{j}$,
$j=1,\dots,l$. 
\end{lem}

\begin{proof}
We must show that if $A+Be^{i\varphi}$ is not Hurwitz for some $\varphi$,
then either the case II or III is realized. 

Assume that $A+Be^{i\varphi_{0}}$ is not Hurwitz, i.e.,
\begin{equation}
\det\left(-\lambda_{1}\text{I}+A+Be^{i\varphi_{0}}\right)=0\,\,\text{with}\,\,\,\Re(\lambda_{1})\ge0.\label{eq:non-hurwitz}
\end{equation}

Consider the function 
\[
Q(\lambda,z)=\det\left(-\lambda\text{I}+A+zBe^{i\varphi_{0}}\right),
\]
which is a polynomial in $\lambda$. There exists a continuous branch
of complex roots $\lambda(z)$, $z\in\mathbb{C}$, which solves the
polynomial $Q(\lambda(z),z)=0$ and satisfies $\lambda(1)=\lambda_{1}$,
$\Re(\lambda(1))\ge0$. Moreover, $\Re\left(\lambda(0)\right)<0$
due to the fact that $A$ is Hurwitz. Hence, due to continuity of
$\lambda(z)$, there exists $\hat{z}$ with $|\hat{z}|\le1$ such
that $\lambda(\hat{z})=i\hat{\omega}$ and $\Re\left(\lambda(z)\right)<0$
for all $|z|<|\hat{z}|$. That is, we obtain
\begin{align}
Q(i\hat{\omega},\hat{z}) & =\det\left(-i\hat{\omega}\text{I}+A+\hat{z}Be^{i\varphi_{0}}\right)=0,\quad\left|\hat{z}\right|\le1,\label{eq:polyxxx}\\
 & \Re\left(\lambda(z)\right)<0\,\,\,\text{for\,\,all}\,\,|z|<|\hat{z}|.\label{eq:56}
\end{align}

Consider the case $\left|\hat{z}\right|=1$ and denote $\hat{z}=e^{i\hat{\varphi}}$.
For convenience, we rewrite Eqs.~\eqref{eq:polyxxx}--\eqref{eq:56}
for this case: 
\begin{align}
 & \det\left(-i\hat{\omega}\text{I}+A+Be^{i\left(\varphi_{0}+\hat{\varphi}\right)}\right)=0,\label{eq:polyxxx-1}\\
 & \Re\left(\lambda(z)\right)<0\,\,\,\text{for\,\,all}\,\,|z|<1.\label{eq:56-1}
\end{align}
Due to \eqref{eq:56-1}, by continuity, we obtain $\Re\left(\lambda(e^{i\varphi})\right)\le0$
for all $\varphi$. Hence, it holds that either $i\hat{\omega}\in\sigma\left(A+Be^{i\left(\varphi_{0}+\hat{\varphi}\right)}\right)$
or $A+Be^{i\varphi}$ is Hurwitz for all $\varphi\ne\varphi_{0}+\hat{\varphi}$.
There can be up to $n$ isolated pairs $(\hat{\omega},\hat{\varphi})$
satisfying \eqref{eq:polyxxx-1}. 

If there are $\hat{\omega}\ne0$ among the solutions of \eqref{eq:polyxxx-1},
then we immediately obtain the case II of Lemma with $\tilde{\omega}=\hat{\omega}$
and $\tilde{\varphi}=\varphi_{0}+\hat{\varphi}$. If there are only
zero values $\hat{\omega}=0$, we obtain the case III of Lemma with
$\tilde{\varphi}=\varphi_{0}+\bar{\varphi}$. 

Consider the case $\left|\hat{z}\right|<1$ and the function 
\[
Q(i\omega,z)=\det\left(-i\omega\text{I}+A+zBe^{i\varphi}\right)
\]
as a polynomial in $z$. We have, in particular, from \eqref{eq:polyxxx},
that $z(\hat{\omega})=\hat{z}$, $\left|\hat{z}\right|<1$. The polynomial
$Q(i\hat{\omega},z)$ depends non-trivially on $z$, i.e., some coefficient
of this polynomial does not vanish. Indeed, otherwise we obtain $\det\left(-i\hat{\omega}\text{I}+A\right)=Q(i\hat{\omega},0)=Q(i\hat{\omega},\hat{z})=0$,
which contradicts the assumption that $A$ is Hurwitz. Therefore,
there exists a branch of complex roots $z(\omega)$ of $Q(i\omega,z)$,
which depends continuously on $\omega$, and $z(\hat{\omega})=\hat{z}$,
$\left|\hat{z}\right|<1$. Moreover, it is easy to see that $\left|z\right|\to\infty$
as $|\omega|\to\infty$. Due to continuity, there exist $\tilde{\omega}_{1}\in(\hat{\omega},\infty)$
and $\tilde{\omega}_{2}\in(-\infty,\hat{\omega})$ such that $\left|z(\tilde{\omega}_{1,2})\right|=1.$
The two points $\tilde{\omega}_{1,2}$ cannot be zero at the same
time. Let $\tilde{\omega}$ be such nonzero point. Therefore, we have
shown that $i\tilde{\omega}\in\sigma\left(A+Be^{i\tilde{\varphi}}\right)$
with $\tilde{\varphi}=\varphi_{0}+\arg z(\tilde{\omega})$. This corresponds
to the case II. 
\end{proof}

\subsection{Absolute stability conditions in terms of extended singular maps
\eqref{eq:map-omega}}

The following lemma shows that condition (A2.2) {[}almost Hurwitz
S($\Phi$){]} can be recast in terms of a spectral radius criterion. 
\begin{lem}
\label{lem:omega-phi}Assume $A_{0}$ is Hurwitz. Then the following
statements hold:

(I) $e^{-i\varphi}\in\sigma\left(\left(i\omega\text{I}-A_{0}\right)^{-1}A_{1}\right)$
if and only if $i\omega\in\sigma\left(A_{0}+A_{1}e^{i\varphi}\right)$.

(II) $\rho\left[\left(i\omega\text{I}-A_{0}\right)^{-1}A_{1}\right]<1$
for all $\omega\in\mathbb{R}$ if and only if $A_{0}+A_{1}e^{i\varphi}$
is Hurwitz for all $\varphi\in S^{1}$.

(III) $\rho\left[\left(i\omega\text{I}-A_{0}\right)^{-1}A_{1}\right]<1$
for all $\omega\ne0$ if and only if $\Re\left[\sigma\left(A_{0}+A_{1}e^{i\varphi}\right)\setminus\{0\}\right]<0$
for all $\varphi\in S^{1}$. 
\end{lem}

\begin{proof}
(I) follows from the equivalent expressions 
\begin{eqnarray*}
\det\left[e^{-i\varphi}\text{I}-\left(i\omega-A_{0}\right)^{-1}A_{1}\right] & = & 0,\\
\det\left[i\omega\text{I}-A_{0}-A_{1}e^{i\varphi}\right] & = & 0.
\end{eqnarray*}

(II) Assume $\rho\left[\left(i\omega-A_{0}\right)^{-1}A_{1}\right]<1$
for all $\omega\in\mathbb{R}$. Then (I) implies that the matrix $A_{0}+A_{1}e^{i\varphi}$
possesses no purely imaginary eigenvalues. Since $A_{0}$ is Hurwitz,
Lemma~\ref{lem:new-lemma-1} implies that $A_{0}+A_{1}e^{i\varphi}$
is also Hurwitz. 

To prove the converse, assume $A_{0}+A_{1}e^{i\varphi}$ is Hurwitz
and let us show that the condition $\rho\left[\left(i\omega-A_{0}\right)^{-1}A_{1}\right]<1$
holds for all $\omega$. It clearly holds for sufficiently large $\omega$.
If, it fails for some $\omega,$ then, there must exist $\omega=\omega_{0}$
such that $e^{i\varphi}\in\sigma\left[\left(i\omega_{0}-A_{0}\right)^{-1}A_{1}\right]=1$.
However, the statement (I) implies that $A_{0}+A_{1}e^{i\varphi}$
is not Hurwitz. 

(III) This statement follows from the continuity of eigenvalues as
functions of $\omega$ and statements (I) and (II). 
\end{proof}
With Lemma~\ref{lem:omega-phi} we obtain that for systems with one
delay the criteria for absolute stability from Theorems \ref{thm:AS-general}
and \ref{thm:AS-general-1} can be equivalently reformulated as follows.
\begin{lem}
\label{thm:1delay-2} System (\ref{eq:lin-1}) with a single delay
is absolutely stable if and only if the following conditions are satisfied:

\emph{(A): $A_{0}$ is Hurwitz.}

\emph{(B): $A_{0}+A_{1}$ is nonsingular}

\emph{(C)
\begin{equation}
\rho\left(\left(i\omega I-A_{0}\right)^{-1}A_{1}\right)<1\,\text{ for all }\omega\ne0.\label{eq:condition-C-with-omega}
\end{equation}
}
\end{lem}

Lemma~\ref{thm:1delay-2} implies immediately the statement of Corollary~\ref{cor:singular}. 

\subsection{Absolute stability and asymptotic spectrum\label{subsec:Absolute-stability-and}}

In view of Corollary \ref{thm:absolut-vs-large}, the stability for
large delays and the absolute stability are equivalent. In this section,
we discuss this relation in more details. The spectrum of DDEs can
be well approximated in the limit $\tau\to\infty$. More specifically,
the spectrum of DDEs with one large discrete delay can be generically
divided into two parts \cite{Lichtner2011,Sieber2013,Yanchuk2015b}:

(i) The \emph{strongly unstable} part $\mathcal{S}_{\text{su}}$,
which is approximated by the unstable spectrum of $A_{0}$, i.e. $\sigma(A_{0})$
with $\Re\sigma\left(A_{0}\right)>0$, and

(ii) the \emph{pseudo-continuous} spectrum $\mathcal{S}_{\text{pc}}$,
which is approximated by the curves 
\begin{equation}
B_{1}=\left\{ z\in\mathbb{C}:\,z=\frac{1}{\tau}\gamma_{j}(\omega)+i\omega,\quad\omega\in\mathbb{R},\quad j=1,\dots,m_{1}\right\} \label{eq:AS}
\end{equation}
in the complex plane. The functions $\gamma_{j}(\omega)$ are given
by\emph{ 
\begin{equation}
\gamma_{j}(\omega)=-\ln\left|Y_{j}(\omega)\right|,\label{eq:gamma}
\end{equation}
}where $Y_{j}(\omega)$, $j=1,\text{rank}A_{1}$, are the roots of
the \emph{spectral polynomial 
\begin{equation}
p(i\omega,Y)=\det\left[i\omega\cdot\mathrm{I}-A_{0}-A_{1}Y\right].\label{eq:SpPo}
\end{equation}
}In particular, the functions $\gamma_{j}(\omega)$ are continuous
except for the isolated points $\omega_{s}$ where $\lim_{\omega\to\omega_{s}}\gamma_{j}(\omega)=\pm\infty$.
The points $\omega_{s}$ where $\lim_{\omega\to\omega_{s}}\gamma_{j}(\omega)=+\infty$
are determined by the condition $i\omega_{s}\in\sigma(A_{0})$. Clearly,
if such a point exists, it leads to an instability for large delays. 
\begin{defn}
The set (\ref{eq:AS}) is called the \emph{asymptotic continuous spectrum}
\cite{Lichtner2011}. 
\end{defn}

We are now ready to provide an interpretation of the conditions of
Lemma~\ref{thm:1delay-2} in terms of the asymptotic spectrum. Condition
(A), i.e. $\Re\left(\sigma(A_{0})\right)<0$, is also the same as
(A1.1) {[}instantaneous stability{]} in Theorems \ref{thm:AS-general}
and \ref{thm:AS-general-1}. It guarantees that, first, the strongly
unstable spectrum is absent, and, second, the asymptotic continuous
spectrum possess no singularities, see~Fig.~\ref{fig:Spectrum-scalar}.
Condition (B) is the same as (A1.2) {[}nonsingular $S(0)${]} in Theorems
\ref{thm:AS-general} and \ref{thm:AS-general-1}, and it excludes
the existence of the trivial eigenvalue $\lambda=0$. Condition (C)
guarantees that the asymptotic continuous spectrum is located in the
open left half of the complex plane $\Re(\lambda)<0$, possibly touching
the origin, see Fig.~\ref{fig:Spectrum-scalar}. Indeed, let $\mu$
be an eigenvalue of $\left(i\omega I-A_{0}\right)^{-1}A_{1}$. Then
the condition (C) from Theorem~\ref{thm:1delay-2} can be rewritten
as 
\[
\det\left[i\omega I-A_{0}-\mu^{-1}A_{1}\right]=0,\quad\left|\mu\right|<1,\,\,\omega\ne0,
\]
which means that all roots $Y_{j}(\omega)$, $j=1,\text{rank}A_{1}$,
of the spectral polynomial (\ref{eq:SpPo}) satisfy $|Y_{j}(\omega)|>1$
for $\omega\ne0$, implying $\gamma_{j}(\omega)<0$ for all $\omega\ne0$.

\subsection{Scalar DDEs with one delay\label{subsec:Scalar-DDEs-with}}

As a simple illustration, we present the complex scalar DDE 
\begin{equation}
\dot{x}(t)=a_{0}x(t)+a_{1}x(t-\tau)\label{eq:scalar}
\end{equation}
with the characteristic equation 
\begin{equation}
\lambda-a_{0}-a_{1}e^{-\lambda\tau}=0,\label{eq:scalar-chareq}
\end{equation}
$a_{0},a_{1}\in\mathbb{C}$. For this case, the real part of the asymptotic
continuous spectrum has a unique global maximum at $\omega=\Im(a_{0})$.
Indeed, the spectral polynomial (\ref{eq:SpPo}) has one root $Y=(i\omega-a_{0})/a_{1}$
leading to 
\[
\gamma(\omega)=-\frac{1}{2}\ln\left(\left(\omega-\Im(a_{0})\right)^{2}+\left(\Re(a_{0})\right)^{2}\right)+\ln\left|a_{1}\right|
\]
with 
\begin{equation}
\max_{\omega\in\mathbb{R}}\gamma(\omega)=\gamma(\Im(a_{0}))=-\ln\left|\frac{\Re(a_{0})}{a_{1}}\right|.\label{eq:maxgamma}
\end{equation}
The absolute stability criterion for (\ref{eq:scalar}) follows form
Theorem~\ref{thm:Scalar}: 
\begin{cor}
\label{cor:scalar-onedelay}The DDE (\ref{eq:scalar}) is absolutely
stable if and only if the following conditions are satisfied: 
\begin{equation}
\begin{cases}
\Re(a_{0})+\left|a_{1}\right|<0, & \Im(a_{0})\ne0\\
a_{0}+\left|a_{1}\right|\le0\,\,\text{and}\,\,a_{0}+a_{1}\ne0, & \Im(a_{0})=0.
\end{cases}\label{eq:cond-a2}
\end{equation}
\end{cor}

It is easy to see that the conditions of the Corollary~\ref{cor:scalar-onedelay}
imply the stability of the asymptotic spectrum. The asymptotic continuous
spectrum is allowed to touch the imaginary axis at the origin, and
this is the case when $a_{0}+\left|a_{1}\right|=0$, however, the
additional condition $a_{0}+a_{2}\ne0$ forbids the appearance of
the trivial eigenvalue.

Finally, we notice that the asymptotic continuous spectrum crosses
the imaginary axis at the points 
\[
\omega_{H}=\Im(a_{1})\pm\sqrt{\left|a_{2}\right|^{2}-\left(\Re(a_{1})\right)^{2}}
\]
in the unstable case. The values $\omega_{H}$ are possible frequencies
of the Hopf bifurcations in corresponding nonlinear systems.

\begin{figure}
\includegraphics[width=1\textwidth]{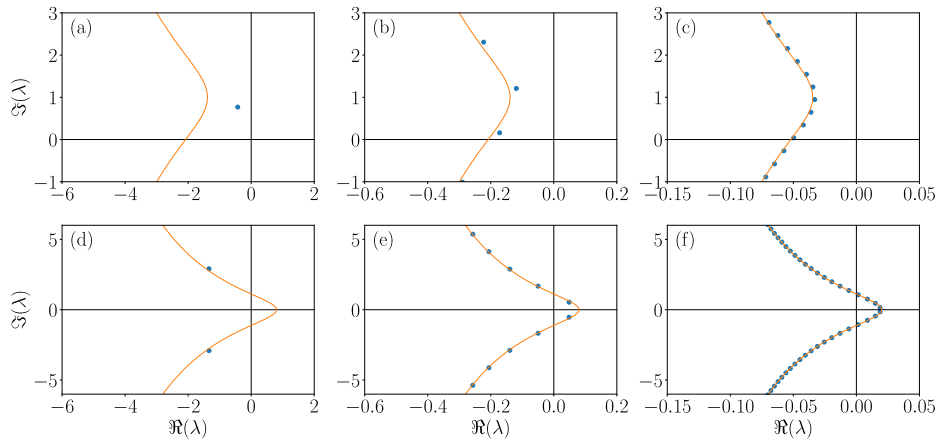}\caption{\label{fig:Spectrum-scalar} Spectrum (blue points) and the asymptotic
continuous spectrum (\ref{eq:AS}) (orange lines) of the scalar system
(\ref{eq:scalar}). The upper panels (a-c) correspond to an absolutely
stable case for the parameter values $a_{0}=-1+i$ and $a_{1}=0.5$.
Time-delay is increasing from (a) to (c): $\tau=0.5$ (a), $\tau=5$
(b), and $\tau=20$ (c). Similarly, the lower panels (d-f) illustrate
a case without absolute stability for the parameter values $a_{0}=-1$
and $a_{1}=-1.5$. Time-delays are: $\tau=0.5$ (d); $\tau=5$ (e),
and $\tau=20$ (f).}
\end{figure}

\subsection{Discussion of Corollary \ref{cor:singular}}

\label{sec:singular} 

Here we explain the physical meaning of the extended singular map
\eqref{eq:map-omega}, which appears in Corollary \ref{cor:singular}
and determines the absolute stability. According to the corollary
assumptions, it must be exponentially stable for all $\omega\ne0$
and any $\varphi\in\mathbb{T}$. The map (\ref{eq:map-omega}) can
be obtained form the single-delay DDE by substituting $x(t)=y(t)e^{i\omega t/\varepsilon}$,
$\varphi=$$\omega/\varepsilon$, and formally neglecting the term
$\varepsilon\dot{y}$. From the physical point of view, equation (\ref{eq:map-omega})
regulates the amplification or damping of rapid oscillations with
frequency $\omega/\varepsilon$. By rescaling the time back to the
original form, these are frequencies $\omega$. 

\section{Multiple delays}

\label{multipleDelays}

\subsection{Equivalence of absolute stability and asymptotic stability for hierarchically
large delays}

In this section we show that the criterium for the absolute stability
for arbitrary positive delays is equivalent to the stability for hierarchically
large time-delays, i.e., the asymptotic stability for $1\ll\tau_{1}\ll\cdots\ll\tau_{m}$.
Such an equivalence is a generalization of Corollary~\ref{thm:absolut-vs-large}
for one large delay. Interestingly, due to the symmetry of the conditions
for the absolute stability with respect to the numbering of the delays,
the order $\tau_{k}$ in this case does not play any role.

For the proof, we will need several Lemmas.
\begin{lem}
\label{lem:new-lemma-2} Let $A\in\mathbb{C}^{n\times n}$ and $i\omega_{0}\in\sigma(A)$.
Then, for any $B\in\mathbb{C}^{n\times n}$, one of the following
two mutually exclusive cases occurs: \\
I. There exist $\tilde{\omega}\ne0$ and $\tilde{\varphi}\in\mathbb{T}$
such that $i\tilde{\omega}\in\sigma\left(A+Be^{i\tilde{\varphi}}\right)$.
\\
II. $\omega_{0}=0$ and $0\in\sigma\left(A+Be^{i\varphi}\right)$
for all $\varphi\in\mathbb{T}$.
\end{lem}

\begin{proof}
Consider the function 
\[
Q(i\omega,z)=\det\left(-i\omega\text{I}+A+zB\right).
\]
The Lemma's assumption implies $Q(i\omega_{0},0)=0$. Two cases are
possible:\\
1. The polynomial $Q(i\omega_{0},z)$ does not depend on $z$. In
such a case, for arbitrary $z$, we have $Q(i\omega_{0},z)=Q(i\omega_{0},0)=0$.
In particular, it holds $Q(i\omega_{0},e^{i\varphi})=0$, hence, $i\omega_{0}\in\sigma\left(A+Be^{i\varphi}\right)$
for all $\varphi$. If $\omega_{0}=0$, then the case II is realized.
For $\omega_{0}\ne0$, the case I is realized. \\
2) The polynomial $Q(i\omega_{0},z)$ depends non-trivially on $z$.
Then, there exists a branch of complex roots $z(\omega)$ solving
$Q(i\omega,z(\omega))=0$, which depends continuously on $\omega$,
and $z(\omega_{0})=0$. Moreover, it holds $\left|z(\omega)\right|\to\infty$
as $\left|\omega\right|\to\infty$. Due to continuity, there exist
$\tilde{\omega}_{1}\in(\omega_{0},\infty)$ and $\tilde{\omega}_{2}\in(-\infty,\omega_{0})$
such that $\left|z(\tilde{\omega}_{1,2})\right|=1.$ Hence, we obtain
$i\tilde{\omega}_{1,2}\in\sigma\left(A+Be^{i\tilde{\varphi}_{1,2}}\right)$
with $\tilde{\varphi}_{1,2}=\arg(z(\tilde{\omega}_{1,2}))$. Since
the two values $\tilde{\omega}_{1,2}$ cannot be zero simultaneously,
we obtain the case I of the Lemma. 
\end{proof}
\begin{lem}
\label{lem:new-lemma-many}Let $A_{0}\in\mathbb{C}^{n\times n}$ be
Hurwitz and $A_{k}\in\mathbb{C}^{n\times n}$, $k=1,\dots,m$. Then,
one of the following three mutually exclusive cases occurs: \\
I. $S(\text{\ensuremath{\Phi}})$ is Hurwitz for all $\Phi\in\mathbb{T}^{m}$;\\
II. There exist $\tilde{\omega}\ne0$ and $\tilde{\Phi}\in\mathbb{T}^{m}$
such that $i\tilde{\omega}\in\sigma\left(S(\tilde{\text{\ensuremath{\Phi}}})\right)$;\\
III. There exists a nonempty set $\mathcal{T}_{0}\subset\mathbb{T}^{m}$,
$\mathcal{T}_{0}\ne\mathbb{T}^{m}$, such that $0\in\sigma\left(S(\text{\ensuremath{\Phi}})\right)$
for $\Phi\in\mathcal{T}_{0}$, and $S(\text{\ensuremath{\Phi}})$
is Hurwitz for all $\Phi\in\mathbb{T}^{m}\setminus\mathcal{T}_{0}$. 
\end{lem}

\begin{proof}
The proof follows from the consecutive application Lemmas~\ref{lem:new-lemma-1}
and \ref{lem:new-lemma-2} to the matrices 
\begin{equation}
M_{r}=A_{0}+\sum_{k=1}^{r}A_{k}e^{i\phi_{k}},r=0,...,m,\label{eq:matrices}
\end{equation}
where $M_{r-1}$, $r=1,..,m$, plays the role of $A$ and $A_{r}$
plays the role of $B$. Note that in this way Case I of Lemma \ref{lem:new-lemma-1}
transfers the Hurwitz property to the next level $r$, while Case
II of Lemma \ref{lem:new-lemma-1} provides a resonance, which is
then by Lemma \ref{lem:new-lemma-2} transfers to the next level.
Case III of Lemma \ref{lem:new-lemma-1} detects a zero eigenvalue,
which is transferred by Case II of Lemma \ref{lem:new-lemma-2}. By
considering all possible logical chains, wee see that I-III are the
only possibilities that can be realized.
\begin{itemize}
\item [\textbf{I:}]Case I of Lemma~\ref{lem:new-lemma-1} for all $r=1,\dots,m$.
In this case, all matrices are Hurwitz for all $\Phi$.
\item [\textbf{II:}]Case I of Lemma~\ref{lem:new-lemma-1}, followed by
Case II of Lemma~\ref{lem:new-lemma-1}, possibly followed by Case
I of Lemma~\ref{lem:new-lemma-2}. Here, we have $i\omega\in\sigma\left(M_{r}\right)$
for some $r\le m$, $\omega\ne0$. Then the sequential application
of Lemma~\ref{lem:new-lemma-2} $m-r$ times leads to $i\tilde{\omega}\in\sigma(S(\tilde{\Phi}))$
with $\tilde{\omega}\ne0$ and some $\tilde{\Phi}$. 
\item [\textbf{II:}]Case I of Lemma~\ref{lem:new-lemma-1}, followed by
Case III of Lemma~\ref{lem:new-lemma-1}, followed by Case I of Lemma~\ref{lem:new-lemma-2}.
Here, the matrix $M_{r}$ contains zero eigenvalue for some $\Phi$
and otherwise it is Hurwitz for all other $\Phi$. At some further
application of Lemma~\ref{lem:new-lemma-2} on some level $r_{1}>r$,
there appears a resonance $\omega\ne0$ such that $i\omega\in\sigma\left(A_{0}+\sum_{k=1}^{r_{1}}A_{k}e^{i\tilde{\varphi_{k}}}\right)$,
$r<r_{1}\le m$. Further application of Lemma~\ref{lem:new-lemma-2}
$m-r_{1}$ times leads to the statement II of this Lemma.
\item [\textbf{III:}]Case I of Lemma~\ref{lem:new-lemma-1}, followed
by Case III of Lemma~\ref{lem:new-lemma-1}, followed by Case II
of Lemma~\ref{lem:new-lemma-2}. Similarly to the previous case,
some matrix $M_{r}$ contains zero eigenvalue and otherwise it is
Hurwitz for all other $\Phi$. At some further applications of Lemma~\ref{lem:new-lemma-2},
only case II of Lemma~\ref{lem:new-lemma-2} is realized. We must
only show that $\mathcal{T}_{0}\ne\mathbb{T}^{m}$. Indeed, assuming
opposite, we have $0\in S(\Phi)$ for all $\Phi$, which implies $0\in A_{0}$
and contradicts the assumption of $A_{0}$ Hurwitz. 
\item [\textbf{II:}]Case I of Lemma~\ref{lem:new-lemma-1}, followed by
Case III of Lemma~\ref{lem:new-lemma-1}, followed by Case II of
Lemma~\ref{lem:new-lemma-2}, followed by Case I of Lemma~\ref{lem:new-lemma-2}.
This logical chain is similar to the previous one, with only difference
that the case I of Lemma~\ref{lem:new-lemma-2} is realized at some
later iteration.
\end{itemize}
\end{proof}
\begin{lem}
\label{lem:new-A0-iomega0}Let $A_{k}\in\mathbb{C}^{n\times n}$,
$k=1,\dots,m$, and $i\omega_{0}\in\sigma(A_{0})$. Then, one of the
following two mutually exclusive cases occurs:\\
I. There exist $\tilde{\omega}\ne0$ and $\tilde{\Phi}$ such that
$i\tilde{\omega}\in\sigma(S(\tilde{\Phi}))$;\\
II. $\omega_{0}=0$ and $0\in\sigma(S(\text{\ensuremath{\Phi}}))$
for all $\Phi\in\mathbb{T}^{m}$.
\end{lem}

\begin{proof}
The proof follows from the sequential application of Lemma~\ref{lem:new-lemma-2}
in a similar way as above.
\end{proof}
\begin{lem}
[\textbf{Reappearance of resonances}]\label{lem:new-hierar-lemma}Let
$A_{k}\in\mathbb{C}^{n\times n}$, $k=0,\dots,m$, and $i\omega_{0}\in\sigma(S(\Phi))$,
$\omega_{0}\ne0$. Then, it holds 
\begin{equation}
\det\left[-i\omega_{0}\text{I}+A_{0}+\sum_{k=1}^{m}A_{k}e^{-i\omega_{0}\tau_{k}}\right]=0\label{eq:cheq-omega0}
\end{equation}
with 
\begin{equation}
\tau_{k}=\frac{2\pi}{\omega_{0}}n_{k}-\frac{\varphi_{k}}{\omega_{0}},\quad n_{k}\in\mathbb{Z}.\label{eq:delays}
\end{equation}
That is, $i\omega_{0}$ solves the characteristic equation \eqref{eq:cheq-1}
for countably many time-delays \eqref{eq:delays}. \\
In particular, among these time-delays, one can choose the set $\left\{ \tau_{1},\dots,\tau_{m}\right\} $
of hierarchically large delays, which satisfy the condition \eqref{eq:delays-hier}
with arbitrary small $\varepsilon>0$. Such delays are hierarchically
ordered so that $\tau_{k}/\tau_{k+1}=\varepsilon(\nu_{k}/\nu_{k+1})$. 
\end{lem}

\begin{proof}
The fact that Eq.~\eqref{eq:cheq-omega0} holds for time-delays \eqref{eq:delays}
can be checked by substitution. 

Let us show that time delays can be chosen to be hierarchical, i.e.,
satisfy the condition \eqref{eq:delays-hier} with arbitrary small
$\varepsilon>0$. We denote 
\[
\varepsilon=\frac{1}{\tau_{1}}=\frac{\omega_{0}}{2\pi n_{1}-\varphi_{1}},
\]
which is a small parameter for sufficiently large $n_{1}$. We assume,
in particular, that $n_{1}\gg\omega_{0}$. Such a definition of $\varepsilon$
implies equality (\ref{eq:delays-hier}) for $k=1$.

Let us show that $n_{k}$, and, hence $\tau_{k}$, can be chosen in
such a way that (\ref{eq:delays-hier}) holds for some $\nu_{k}\in[1,1+\varepsilon^{k-1})$.
The equality 
\[
\tau_{k}=\frac{2\pi n_{k}-\varphi_{k}}{\omega_{0}}=\nu_{k}\varepsilon^{-k}
\]
leads to 
\begin{equation}
n_{k}=\frac{\varphi_{k}}{2\pi}+\frac{\omega_{0}\nu_{k}}{2\pi\varepsilon^{k}}.\label{eq:nk-1}
\end{equation}
By increasing $\nu_{k}$ from 1 to $1+\frac{2\pi\varepsilon^{k}}{\left|\omega_{0}\right|}$,
the value of $n_{k}$ in Eq.~(\ref{eq:nk-1}) changes by 1. Hence,
there exists such $\nu_{k}\in[1,1+\frac{2\pi\varepsilon^{k}}{\left|\omega_{0}\right|})$
that $n_{k}$ admits an integer value. Finally, by choosing $\varepsilon$
sufficiently small such that $\frac{2\pi\varepsilon}{\left|\omega_{0}\right|}<1$,
we obtain that $\nu_{k}\in[1,1+\varepsilon^{k-1})$.
\end{proof}
We remark that Lemma~\eqref{lem:new-hierar-lemma} generalizes some
of the statements shown for one delay in \cite{Yanchuk2009}.

\begin{proof}[Proof of Theorem \ref{thm:absolute-vs-hierar}]
It is clear that the absolute stability implies the stability for
hierarchically large time delays. Therefore, it remains to show that
conditions (A1.1) {[}instantaneous stability{]}, (A1.2) {[}nonsingular
$S(0)${]}, and (A1.3) {[}no resonance{]} of Theorem~\ref{thm:AS-general}
are necessary for the stability of the systems with hierarchically
large time delays \eqref{eq:delays-hier}. 

1. \emph{First, we show that }(A1.1) {[}instantaneous stability{]}\emph{
is necessary.} Assume the opposite, i.e., the condition (A1.1) of
Theorem \ref{thm:AS-general} does not hold. Then either $i\omega_{0}\in\sigma\left(A_{0}\right)$
or $\lambda_{0}\in\sigma\left(A_{0}\right)$ with $\Re(\lambda_{0})>0$. 

1a: Consider the case $i\omega_{0}\in\sigma\left(A_{0}\right)$. Then,
Lemma~\ref{lem:new-A0-iomega0} implies that one of the two cases
can occur: 

1aa: $i\tilde{\omega}\in\sigma(S(\tilde{\Phi}))$ with some $\tilde{\omega}\ne0$.
In such a case, Lemma~\ref{lem:new-hierar-lemma} implies that $i\tilde{\omega}$
is a solution of the characteristic equation for hierarchically large
time delays (\ref{eq:delays-hier}). We obtain the contradiction to
the absolute stability and, hence, (A1.1) holds.

1ab: $\omega_{0}=0$ and $0\in\sigma(S(\Phi))$ for all $\Phi\in\mathbb{T}^{m}$.
In particular, it holds $0\in\sigma(S(0))$, which means that $\lambda=0$
is an eigenvalue for arbitrary time-delays. This contradicts the absolute
stability assumption for hierarchically large delays, hence, (A1.1)
holds.

1b: Consider the case $\lambda_{0}\in\sigma\left(A_{0}\right)$ with
$\Re(\lambda_{0})>0$. Let $\tau_{k}=\nu_{k}\varepsilon^{-k}$ be
hierarchically large delays, and the corresponding characteristic
equation 
\begin{equation}
P_{m}(\lambda)=\det\left[-\lambda\text{I}+A_{0}+\sum_{k=1}^{m}A_{k}e^{-\lambda\nu_{k}\varepsilon^{-k}}\right]=0.\label{eq:cheqhi}
\end{equation}
Let $U(\lambda_{0})$ be a sufficiently small open neighborhood of
$\lambda_{0}$ such that it does not contain other eigenvalues of
$A_{0}$, and $\Re(U(\lambda_{0}))>0$. Then, the holomorphic function
$P_{m}(\lambda)$ converges uniformly to $\det\left[-\lambda\text{I}+A_{0}\right]$
for $\varepsilon\to0$. According to the Hurwitz theorem, the characteristic
equation (\ref{eq:cheqhi}) has an unstable root in $\lambda\in U(\lambda_{0})$
for all sufficiently small $\varepsilon$. This contradicts the asymptotic
stability assumption for hierarchically large delays, and, hence,
(A1.1) holds.

2. \emph{We show that }(A1.2) {[}nonsingular $S(0)${]}\emph{ is necessary.}
Assume that the condition (A1.2) of Theorem \ref{thm:AS-general}
does not hold. Then $0\in\sigma(S(0))$ and, hence the characteristic
root $\lambda=0$ solves Eq.~(\ref{eq:cheq-1}) for all delays. This
contradicts the asymptotic stability assumption for hierarchically
large delays, and, hence, (A1.2) holds.

3. \emph{We show that }(A1.3) {[}no resonance{]}\emph{ is necessary
for the stability of systems with hierarchically large time delays.}
Assume (A1.3) does not hold. Then there exists 
\[
i\omega_{0}\in\sigma\left(S(\Phi)\right),\quad\omega_{0}\ne0.
\]
Lemma~\ref{lem:new-hierar-lemma} implies that there are hierarchically
large time delays, for which there exists the eigenvalue $i\omega_{0}$.
This contradicts the asymptotic stability assumption and, hence, (A1.3)
holds.
\end{proof}

\subsection{Asymptotic spectrum for multiple hierarchically large delays and
its relation to the conditions for absolute stability}

Let us briefly review some concepts for the spectrum of systems with
hierarchically large time-delays $\tau_{k}=\nu_{k}\varepsilon^{-k}$
from \cite{Ruschel2021}. This spectrum can be generically divided
into $m+1$ parts corresponding to different timescales:

(i) The \emph{strongly unstable} part $\mathcal{S}_{\text{su}}$,
which is approximated by the unstable spectrum of $A_{0}$, i.e. $\sigma(A_{0})$
with $\Re\left(A_{0}\right)>0$.

(ii) The asymptotic continuous spectrum on different timescales can
be described by the following sets 
\begin{equation}
B_{k,j}=\left\{ z\in\mathbb{C}:-\ln\left|Y_{k,j}(\omega,\varphi_{1},\dots,\varphi_{k-1})\right|\varepsilon^{k}+i\omega,\,\,\omega\in\mathbb{R}\right\} ,\label{eq:spectralmanifolds}
\end{equation}
where $k=1,\dots,m$. The functions $Y_{k,j}(\omega,\varphi_{1},\dots,\varphi_{k-1})$
are the $j$-th roots of the \emph{spectral polynomial 
\begin{equation}
P_{k}(\omega,\varphi_{1},\dots,\varphi_{k-1},Y)=\det\left[i\omega\cdot\mathrm{I}-A_{0}-\sum_{l=1}^{k-1}A_{l}e^{i\varphi_{j}}-A_{k}Y\right],\label{eq:SpPo-1-1}
\end{equation}
}where the index $j$ numbers the roots. The sets $B_{k,j}$ correspond
to the eigenvalues with the real parts converging to zero as $\varepsilon^{k}$.
For $m=1$, the the sets $B_{1,j}$ contain the asymptotic continuous
spectrum of systems with one large delay $\tau_{1}$.

In the non-degenerate case of $\det A_{m}\ne0$, the asymptotic spectrum
has the form 
\[
\mathcal{S}_{\text{su}}\bigcup\Biggl[\bigcup_{\substack{k=1,\dots,m-1\\
j=1,\dots,\text{rank}A_{k}
}
}B_{k,j}^{+}\Biggr]\bigcup\Biggl[\bigcup_{j=1}^{\text{rank}A_{m}}B_{m,j}\Biggr],
\]
where $B_{k.j}^{+}=B_{k,j}\bigcap\left\{ z:\,\Re z>0\right\} $. That
is, for all spectral components that correspond to the convergence
of real parts as $\mathcal{O}(\varepsilon^{k})$ with $k=1,\dots,m-1$,
only the unstable part is included. The stable part of the asymptotic
continuous spectrum can contain only $B_{m,j}$, which has the slowest
convergence $\varepsilon^{m}$ of the real parts to zero. This implies
that the destabilization of the system with hierarchical delays with
$\det A_{m}\ne0$ can occur only due to some $B_{m,j}$ spectral component,
which is caused by the largest delay $\tau_{m}$. In a degenerate
case of $\det A_{m}=0$, stable parts of other spectral components
may appear as well, see more details in \cite{Lichtner2011,Yanchuk2014,Yanchuk2017,Ruschel2021}.

Taking into account different part of the asymptotic spectra, we can
interpret the role of the conditions of Theorems~\ref{thm:AS-general}
and \ref{thm:AS-general-1} for the spectrum of systems with hierarchical
time delays. Condition (A1.1) {[}instantaneous stability{]} guarantees
the absence of the strongly unstable spectrum. Condition (A1.2) {[}nonsingular
$S(0)${]} guarantees the absence of the zero eigenvalue. Conditions
(A1.3) {[}no resonance{]} and (A2.2) {[}almost Hurwitz S($\Phi$){]}
guarantee that the asymptotic continuous spectrum is stable and do
not cross the imaginary axis. 

\subsection{Illustration in the case of two delays\label{subsec:Illustration-in-the}}

Figure~\ref{fig:Spectrum-scalar-1} illustrates the spectrum of the
scalar DDE with two delays 
\begin{equation}
\dot{x}(t)=a_{0}x(t)+a_{1}x(t-\tau_{1})+a_{2}x(t-\tau_{2}).\label{eq:scalar-1}
\end{equation}
In particular, Figs.~\ref{fig:Spectrum-scalar-1}(a)-(c) show an
absolutely stable case for different values of time-delays. With the
increasing of the delays, the spectrum fills certain regions of the
complex plane but stays stable. Figs.~\ref{fig:Spectrum-scalar-1}(d)-(f)
illustrate the case without absolute stability. One can observes a
stability for small delays and destabilization with the increasing
of the delays.

\begin{figure}
\includegraphics[width=1\textwidth]{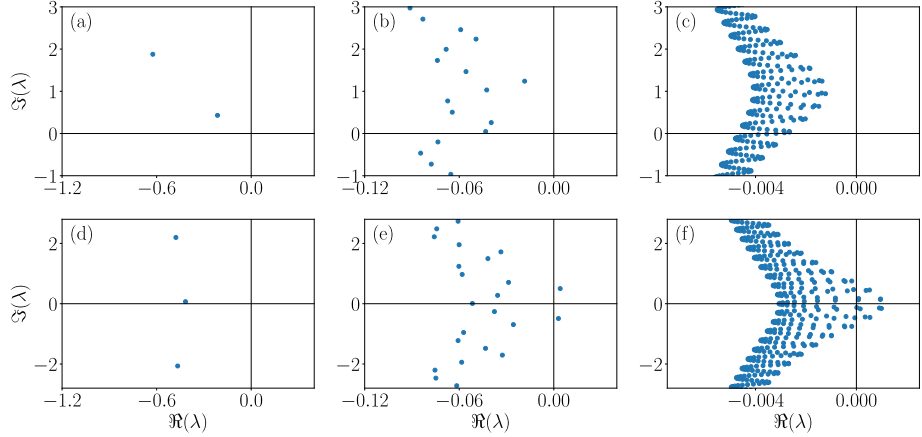}\caption{\label{fig:Spectrum-scalar-1} Spectrum (blue points) of the scalar
system (\ref{eq:scalar-1}) with two delays. The upper panel (a-c)
corresponds to an absolutely stable case for the parameter values
$a_{0}=-1+i$, $a_{1}=0.5$, and $a_{2}=0.3$. Time-delay is increasing
from (a) to (c): $\tau_{1}=0.5$, $\tau_{2}=2.5$ (a), $\tau_{1}=5$,
$\tau_{2}=25$ (b), and $\tau_{1}=20$, $\tau_{2}=400$ (c). The lower
panel (d-f) illustrates the case without absolute stability for $a_{0}=-1$,
$a_{1}=-0.7$, and $a_{2}=0.5+0.1i$. Time-delays are: $\tau_{1}=0.5$,
$\tau_{2}=2.5$ (d), $\tau_{1}=5$, $\tau_{2}=25$ (e), and $\tau_{1}=20$,
$\tau_{2}=400$ (f).}
\end{figure}

\section{Proof of Theorems~\ref{thm:AS-general} and \ref{thm:AS-general-1}
\label{sec:Proofs}}
\begin{lem}
\label{lem:last}Let $A\in\mathbb{C}^{n\times n}$ be not Hurwitz
and $0\not\in\sigma(A)$. Then, for any $B\in\mathbb{C}^{n\times n}$
, there exists $\varphi\in\mathbb{T}$ such that the matrix $A+Be^{i\varphi}$
is not Hurwitz and $\lambda\in\sigma(A+Be^{i\varphi})$ with $\lambda\ne0$
and $\Re(\lambda)\ge0$.
\end{lem}

\begin{proof}
1. Consider first the case $i\omega_{0}\in\sigma(A)$, $\omega_{0}\ne0$.
Then Lemma \ref{lem:new-lemma-2} implies that $i\tilde{\omega}\in A+Be^{i\tilde{\varphi}}$,
$\tilde{\omega}\ne0$ for some $\tilde{\varphi}$. Thus, the statement
of the Lemma follows with $\lambda=i\tilde{\omega}$ and $\varphi=\tilde{\varphi}$. 

2. Let $\lambda\in\sigma(A)$ with $\Re(\lambda)>0$. The following
proof uses similar ideas as in the proof of Lemma \ref{lem:new-lemma-1}.
Consider the function 
\[
Q(\lambda,z)=\det\left(-\lambda\text{I}+A+zB\right).
\]
As a polynomial in $\lambda$, it possesses a continuous branch of
roots $\lambda(z)$ such that $\Re(\lambda(0))>0$. Due to continuity
of $\lambda(z)$, two cases are possible:

2a. $\Re(\lambda(z))>0$ for all $z$ with $\left|z\right|\le1$.
In this case, taking $z=e^{i\varphi}$, we obtain that $A+Be^{i\varphi}$
contains an eigenvalue with $\Re(\lambda(z))>0$ for all $\varphi$. 

2b. There exists $\hat{z}$ such that $\lambda(\hat{z})=i\hat{\omega}$
and $\Re\left(\lambda(z)\right)>0$ for all $|z|<|\hat{z}|$. That
is, we obtain
\begin{align}
Q(i\hat{\omega},\hat{z}) & =\det\left(-i\hat{\omega}\text{I}+A+\hat{z}B\right)=0,\quad\left|\hat{z}\right|\le1,\label{eq:polyxxx-2}\\
 & \Re\left(\lambda(z)\right)>0\,\,\,\text{for\,\,all}\,\,|z|<|\hat{z}|.\label{eq:56-2}
\end{align}

2b-i. Consider the case $|\hat{z}|=1$. We denote $\hat{z}=e^{i\hat{\varphi}}$.
If $\hat{\omega}\ne0$, we obtain $i\hat{\omega}\in A+Be^{i\hat{\varphi}}$,
which is needed for the proof. If $\hat{\omega}\ne0$, we observe
that $\Re\left(\lambda(z)\right)\ge0$ for all $z=e^{i\varphi}$.
Moreover, the equality $\lambda(e^{i\varphi})=0$ cannot hold for
all $\varphi$, since, otherwise, $0\in\sigma\left(A+Be^{i\varphi}\right)$
for all $\varphi$, which is only possible for $0\in\sigma(A)$. Therefore,
there exists $\tilde{\varphi}$ such that $\lambda\in\sigma\left(A+Be^{i\tilde{\varphi}}\right)$
with $\lambda\ne0$ and $\Re(\lambda)\ge0$. 

2b-ii. In the case $\left|\hat{z}\right|<1$, consider the function
\[
Q(i\omega,z)=\det\left(-i\omega\text{I}+A+zB\right)
\]
as a polynomial in $z$. It is nontrivial in $z$ at $\omega=\hat{\omega}$
and $z=\hat{z}$, and there exists a continuous branch of roots such
that $z(\hat{\omega})=\hat{z}$, $\left|\hat{z}\right|<1$, and $\left|z(\omega)\right|\to\infty$
as $\left|\omega\right|\to\infty$. By continuity, we obtain the existence
of $\tilde{\omega}_{1,2}$ with $z(\tilde{\omega}_{1,2})=e^{i\tilde{\varphi}_{1,2}}$.
Hence, we have $i\tilde{\omega}_{1,2}\in\sigma\left(A+Be^{i\tilde{\varphi}_{1,2}}\right)$.
Since $\tilde{\omega}_{1}$ and $\tilde{\omega_{2}}$ belong to disjoint
intervals $(-\infty,\hat{\omega})$ and $(\hat{\omega},+\infty)$,
at least one of them is nonzero. 
\end{proof}
\begin{proof}[Proof of Theorems \ref{thm:AS-general} and \ref{thm:AS-general-1}.]

Firstly, we show that the conditions of Theorems \ref{thm:AS-general}
and \ref{thm:AS-general-1} are equivalent. 

1. The conditions (A2.2) {[}almost Hurwitz S($\Phi$){]} and (A1.2)
{[}nonsingular $S(0)${]} imply that $A_{0}$ is Hurwitz, i.e., (A1.1)
{[}instantaneous stability{]} holds. Assume the opposite, i.e., $A_{0}$
is not Hurwitz. 

When $i\omega_{0}\in\sigma(A_{0}),$ Lemma~\ref{lem:new-A0-iomega0}
implies that one of the following two cases occur:

I. There exists $\omega\ne0$ such that $i\omega\in\sigma(S(\Phi))$
for some $\Phi$. This contradicts the condition (A2.2) {[}almost
Hurwitz S($\Phi$){]}. 

II. $\omega_{0}=0$ and $0\in\sigma(S(\Phi))$ for all $\Phi\in\mathbb{T}^{m}$.
Substituting $\Phi=0$, we obtain $0\in\sigma(S(0))$, which contradicts
the condition (A1.2) {[}nonsingular $S(0)${]}. 

Now assume that $\sigma(A_{0})$ does not contain purely imaginary
eigenvalues. Since $A_{0}$ is not Hurwitz, we have $\lambda\in\sigma(A)$
with $\Re(\lambda)>0$. Applying Lemma~\eqref{lem:last} sequentially,
we obtain that there is $\lambda\in\sigma(S(\Phi))$ with some $\Phi$
with $\Re(\lambda)\ge0$ and $\lambda\ne0$. This contradicts to the
condition (A2.2) {[}almost Hurwitz S($\Phi$){]}. 

We have shown that $A_{0}$ is Hurwitz under the assumptions of Theorem~\ref{thm:AS-general-1}.
Let us show that (A2.2) {[}almost Hurwitz S($\Phi$){]} and (A1.3)
{[}no resonance{]} are equivalent when $A_{0}$ is Hurwitz. Applying
Lemma~\ref{lem:new-lemma-many}, one can see that cases I and III
of Lemma~\ref{lem:new-lemma-many} correspond to the condition (A2.2)
{[}almost Hurwitz S($\Phi$){]} of Theorem~\ref{thm:AS-general-1}.
Moreover, the condition (A1.3) {[}no resonance{]} of Theorem~\ref{thm:AS-general}
excludes the case II of Lemma~\ref{lem:new-lemma-many}, hence, it
is also equivalent to the case I or III of Lemma~\ref{lem:new-lemma-many}.
Hence, (A1.3) and (A2.2) are equivalent. 

The following steps (ii)--(iii) prove that (A1.2) {[}nonsingular
$S(0)${]} and (A2.2) {[}almost Hurwitz S($\Phi$){]} are sufficient
for the absolute stability.

2. First notice, that (A2.2) {[}almost Hurwitz S($\Phi$){]} implies
that $S(0)$ is almost Hurwitz, i.e., $S(0)$ is Hurwitz, except for
a possible zero eigenvalue. However, zero eigenvalue is excluded by
the condition (A1.2) {[}nonsingular $S(0)${]}. Hence, $S(0)$ is
Hurwitz. 

The spectrum for $\tau_{k}=0$, $k=1,\dots,m$ coincides with the
spectrum of $S(0)$, which is Hurwitz. Hence, all roots for $\tau_{k}=0$
possess negative real parts. The same also holds for sufficiently
small delays, see e.g. \cite{Smith2010}.

3. Due to continuity of the roots $\lambda$ with respect to $\tau_{k}>0$,
the only possible stability loss for positive delays is through the
crossing of the imaginary axis. Let us assume that $\lambda=i\omega^{*}$
at some $\tau_{k}=\tau_{k}^{*}>0$, $k=1,\dots,m$, and subsequently
show that it leads to a contradiction. Indeed $\lambda=i\omega^{*}$
implies 
\[
i\omega^{*}\in\sigma\left(S(\Phi^{*})\right),\quad\varphi_{k}^{*}=-\omega^{*}\tau_{k}^{*}.
\]
Due to (A2.2) {[}almost Hurwitz S($\Phi$){]}, it holds $\omega^{*}=0$.
However, in this case, $0\in\sigma\left(S(0)\right)$, which contradicts
the assumption (A1.2) {[}nonsingular $S(0)${]}.

Hence, for all positive delays, the roots cannot cross the imaginary
axis and the asymptotic stability holds, i.e. the conditions (A1.2)
and (A2.2) imply the absolute stability.

The following steps (iv)-(vi) prove that (A1.1) {[}instantaneous stability{]},
(A1.2) {[}nonsingular $S(0)${]}, and (A1.3) {[}no resonance{]} are
necessary conditions for the absolute stability. We choose here the
conditions (A1.1), (A1.2), (A1.3) from Theorem~\ref{thm:AS-general},
since they are equivalent to (A1.2) and (A2.2), and they are more
convenient for the proof of necessity. Hence, we assume that absolute
stability holds and show (A1.1), (A1.2), and (A1.3).

(iv) Assume (A1.1) {[}instantaneous stability{]} does not hold, then
there exists $\lambda_{0}\in\sigma(A_{0})$ with $\Re\left(\lambda_{0}\right)\ge0$.

If $\Re\left(\lambda_{0}\right)>0$, consider the case of large delays
$\tau_{k}=\varepsilon^{-1}$. The corresponding characteristic equation
has the form
\begin{equation}
P_{m}(\lambda)=\det\left[-\lambda\text{I}+A_{0}+\sum_{k=1}^{m}A_{k}e^{-\lambda/\varepsilon}\right]=0.\label{eq:cheqhi-1}
\end{equation}
Let $U(\lambda_{0})\subset\mathbb{C}$ be a sufficiently small open
neighborhood of $\lambda_{0}$ such that it does not contain other
eigenvalues of $A_{0}$, and $\Re(U(\lambda_{0}))>0$. Then, the holomorphic
function $P_{m}(\lambda)$ converges uniformly to $\det\left[-\lambda\text{I}+A_{0}\right]$
for $\varepsilon\to0$ on $U(\lambda_{0})$. According to the Hurwitz
theorem, the characteristic equation (\ref{eq:cheqhi-1}) has an unstable
root in $\lambda\in U(\lambda_{0})$ for all sufficiently small $\varepsilon$.
This contradicts to the absolute stability assumption and, hence,
(A1.1) is a necessary condition.

If $i\omega_{0}\in\sigma\left(A_{0}\right)$, Lemma~\ref{lem:new-A0-iomega0}
implies that one of the two cases can occur: 

1. $i\tilde{\omega}\in\sigma\left(S(\tilde{\Phi})\right)$ with some
$\tilde{\omega}\ne0$. In such a case, Lemma~\ref{lem:new-hierar-lemma}
implies that $i\tilde{\omega}$ is a solution of the characteristic
equation for countable number of delays (\ref{eq:delays}). We obtain
the contradiction to the absolute stability and, hence, (A1.1) is
necessary.

2. $\omega_{0}=0$ and $0\in\sigma\left(S(\Phi)\right)$ for all $\Phi\in\mathbb{T}^{m}$.
In particular, it holds $0\in\sigma\left(S(0)\right)$, which means
that $\lambda=0$ is an eigenvalue for arbitrary time-delays. This
contradicts the absolute stability assumption, hence, (A1.1) is necessary.

(v) The necessity of (A1.2) {[}nonsingular $S(0)${]} is evident,
since otherwise there exists a root $\lambda=0$ for all delays.

(vi) We show that (A1.3) {[}no resonance{]} is necessary. Assume the
opposite, i.e., $i\omega_{0}\in\sigma(S(\Phi))$, $\omega_{0}\ne0$
for some $\Phi$. Then, accordingly to Lemma~\ref{lem:new-hierar-lemma},
systems with time-delays (\ref{eq:delays}) possess the eigenvalues
$i\omega_{0}$. This contradicts the absolute stability and proves
that (A1.3) is necessary. 

Finally, let us show the criterion for the absolute hyperbolicity
from Theorem~\ref{thm:AS-general}. We first prove that (A1.2) {[}nonsingular
$S(0)${]} and (A1.3) {[}no resonance{]} imply absolute hyperbolicity.
Assume the opposite, so that there exists a solution $\lambda=i\omega$
of Eq.~(\ref{eq:cheq-1}) for some time delays. Then, if $\omega=0$,
then we obtain the contradiction to (A1.2); if $\omega\ne0$, we obtain
the contradiction to (A1.3) with $\varphi_{k}=-\omega\tau_{k}$. The
backward statement ``absolute hyperbolicity'' $\Rightarrow$ (A1.2)
and (A1.3) is also straightforward. Assuming that (A1.2) or (A1.3)
does not hold, we obtain either $\lambda=0$ or $\lambda=i\omega\ne0$,
respectively. 
\end{proof}

\section{Conclusions}

\label{con}

The obtained conditions for absolute stability determine a class of
linear DDEs, which are asymptotically exponentially stable, independently
on time-delays. Such class of systems can be useful for applications,
where the robustness against time-delays is important. For nonlinear
systems, these conditions exclude the possibility of any bifurcations
at the corresponding equilibrium.

Bifurcations induced by varying time delay are also excluded in the
case of absolute hyperbolicity. Linear systems that do not belong
to one of these two classes have resonances, i.e. purely imaginary
eigenvalues, which occur for countably many resonant delay times in
each delayed argument, and are necessarily unstable for large delays.
Note that such systems may or may not become stable for certain ranges
of small delays. Even systems with strong instabilities for large
delay may become stable for small delay, but only if they have unstable
asymptotic continuous spectrum. This counter-intuitive conclusion
follows from absolute hyperbolicity, which we showed for strongly
unstable systems with stable asymptotic continuous spectrum. 

\textbf{Acknowledgment:} SY was supported by the German Science Foundation
(Deutsche Forschungsgemeinschaft, DFG) {[}project No. 11803875{]}.
TP was supported by a Newton Advanced Fellowsip of the Royal Society
NAF\textbackslash R1\textbackslash 180236, by Serrapilheira Institute
(Grant No. Serra-1709-16124), and FAPESP (grant 2013/07375-0). 

\bibliographystyle{unsrt}
\bibliography{Papers-paper-absolute}

\end{document}